\documentclass[10pt]{amsart}
\usepackage[latin1]{inputenc}
\usepackage[T1]{fontenc}
\usepackage{graphicx}
\usepackage{amsmath}
\usepackage{amssymb}
\usepackage{amsfonts}
\usepackage{amsthm}
\usepackage{cite}
\usepackage{paralist}
\usepackage{array}
\usepackage{longtable}
\allowdisplaybreaks[1]
\usepackage{eurosym}
\usepackage{setspace}
\usepackage{color}
\usepackage[pdfborder={0 0 0}]{hyperref}
\usepackage{comment}

\numberwithin{equation}{section}
\theoremstyle{definition}
\newtheorem{defn}{Definition}[section]
\theoremstyle{plain}
\newtheorem{thm}{Theorem}[section]
\newtheorem{lem}[thm]{Lemma}
\newtheorem{prop}[thm]{Proposition}

\newenvironment{prf}{\begin{proof}[\textbf{Proof}]}{\end{proof}}

\newcommand{\img}{\mathbf{i}}
\newcommand{\Un}[1]{\mathbf{1}_{#1}}

\newcommand{\set}[1]{\left\{#1\right\}}
\newcommand{\norm}[1]{\left\Vert#1\right\Vert}

\newcommand{\cA}{{\mathcal{A}}}
\newcommand{\cB}{{\mathcal{B}}}
\newcommand{\cBp}{{\widehat{\cB}^+}}
\newcommand{\cBd}{{\widehat{\cB}}}

\newcommand{\cD}{{\mathcal{D}}}

\newcommand{\cH}{{\mathcal{H}}}
\newcommand{\cI}{{\mathcal{I}}}

\newcommand{\cK}{{\mathcal{K}}}

\newcommand{\cO}{{\mathcal{O}}}

\newcommand{\cU}{{\mathcal{U}}}

\newcommand{\dC}{{\mathbb{C}}}
\newcommand{\dN}{{\mathbb{N}}}
\newcommand{\dR}{{\mathbb{R}}}

\newcommand{\dZ}{{\mathbb{Z}}}

\newcommand{\gA}{\mathfrak{A}}
\newcommand{\gB}{\mathfrak{B}}

\newcommand{\gee}{\mathbf{e}}
\newcommand{\gog}{\mathfrak{g}}

\newcommand{\Ind}{{\mathrm{Ind}}}

\newcommand{\Res}{{\mathrm{Res}}}

\newcommand{\Orb}{{\mathrm{Orb}}}
\newcommand{\St}{{\mathrm{St}}}
\newcommand{\Id}{{\mathrm{Id}}}
\newcommand{\Rep}{\mathrm{Rep}}

\newcommand{\supp}{\mathrm{supp}}

\newcommand{\abs}[1]{\left\lvert#1\right\rvert}

\title{Induced $*$-representations and $C^*$-envelopes of some quantum $*$-algebras.}

\begin{document}

\author{Philip A. Dowerk}
\address{MPI Leipzig, Germany}
\email{dowerk@mis.mpg.de, p.dowerk@gmx.de}
\thanks{The first author was supported by the International Max Planck Research School for Mathematics in the Sciences (Leipzig)}

\author{Yurii Savchuk}
\address{Universit\"at Leipzig, Mathematisches Institut Johannisgasse 26, 04103 Leipzig, Germany}
\email{savchuk@math.uni-leipzig.de}


\date{\today.}


\begin{abstract}
We consider three quantum algebras: the $q$-oscillator algebra, the Podles' sphere and the $q$-deformed enveloping algebra of $su(2).$ To each of these $*$-algebras we associate certain partial dynamical system and perform the ``Mackey analysis'' of $*$-representations developed in \cite{SS}. As a result we get the description of ``standard'' irreducible $*$-representations. Further, for each of these examples we show the existence of a ``$C^*$-envelope'' which is canonically isomorphic to the covariance $C^*$-algebra of the partial dynamical system. Finally, for the $q$-oscillator algebra and the $q$-deformed $\cU(su(2))$ we show the existence of ``bad'' representations.
\end{abstract}

\maketitle

\section*{Introduction and preliminaries}\label{introduction}

The aim of this paper is to demonstrate a unified approach to the $*$-representation theory of various quantum algebras based on the techniques developed in \cite{SS}. Most of the quantum $*$-algebras (e.g. non-compact quantum groups) possess unbounded $*$-representations. The main problem in the theory of unbounded $*$-representations is to define and classify the ``well-behaved'' $*$-representations of a given $*$-algebra. We recall two classical examples. 

\noindent \textbf{Example.} Let $\gog$ be a finite-dimensional real Lie algebra, $G$ be the corresponding simply connected Lie group and $\cU_\dC(\gog)$ be the complex enveloping $*$-algebra of $\gog.$ A $*$-representation $\pi$ of $\gog$ is called \textit{integrable} if $\pi=dU$ for some unitary representation $U$ of $G.$ If $G\neq \dR$ there exists a $*$-representation of $\gog$ which is not integrable and, moreover, cannot be extended to an integrable representation even in a larger Hilbert space, see \cite{s3}. Already in the case $G=\dR^2,\ \gog=\dC[x_1,x_2]$ the category of all $*$-representations of $\gog$ is in a certain sense ``very large'' as shown in \cite[Section 9]{s}. 

\noindent \textbf{Example.} Let $W_n$ be the $n$-dimensional Weyl algebra. That is, $W_n$ is a complex $*$-algebra generated by self-adjoint elements $p_i,q_i,\ i=1,\dots, n,$ satisfying $[p_i,q_j]=-\delta_{ij}\img,\ [p_i,p_j]=[q_i,q_j]=0.$ A $*$-representation $\pi$ of $W_n$ is called integrable if $P_i=\overline{\pi(p_i)},\ Q_j=\overline{\pi(q_j)},\ i,j=1,\dots, n$ are self-adjoint and the one-parameter unitary groups $e^{\img t_iP_i},\ e^{\img s_jQ_j}$ satisfy the Weyl commutation relations. Already for $W_1$ one can show the existence of ``bad representations'' and show that the category of all $*$-representations is again ``very large'', whereas the only integrable $*$-representations are sums of copies of the Schr\"odinger representation.

We investigate the following three $*$-algebras in details: the $q$-oscillator algebra $\cA_q$ for $q>0,$ the $q$-deformed enveloping algebra $\cU_q(su(2)),\ q>0$ and the Podles' spheres $\cO(S_{qr}^2),\ q\in(0,1),\ r\in(0,\infty).$ The algebras $\cA_q$ and $\cU_q(su(2))$ are deformations of $W_1$ and $\cU_\dC(su(2))$ respectively, however, for both these algebras the notion of ``integrability'' cannot be generalized in a direct way. Instead of this we use the approach from \cite{SS}, which applies to all three algebras $\cA_q,\ \cO(S_{qr}^2),\ \cU(su(2))$ as well as to their classical analogues. Let $\cA$ denote one of these algebras. The basic idea is to find a natural $\dZ$-grading $\cA_k,\ k\in\dZ,$ for $\cA$ such that $\cA_0=:\cB$ is commutative. Further, we define the ``positive'' spectrum $\cBp$ of $\cB$ as the set of those characters $\chi\in\cBd$ which satisfy $\chi(a^*a)\geq 0$ for all $a\in\cA,$ such that $a^*a\in\cB.$ The group grading of $\cA$ defines a structure of a $*$-algebraic bundle in the sense of \cite{fd}, and there is a canonical partial action $\alpha$ of $\dZ$ on $\cBp.$ By means of the partial dynamical system $(\cBp,\dZ,\alpha)$ we
\begin{itemize}
 \item[--] define well-behaved $*$-representations,
 \item[--] show that the irreducible ones naturally correspond to the orbits of $(\cBp,\dZ,\alpha),$
 \item[--] construct the dual partial action $\beta$ on $C_0(\cBp)$ and the partial crossed product $C^*$-algebra $C_0(\cBp)\times_\beta\dZ$ in the sense of \cite{Ex}; using the Woronowicz's theory of affiliated operators, we establish a Morita equivalence between $C_0(\cBp)\times_\beta\dZ$ and $\cA.$
\end{itemize}
It turns out that every irreducible well-behaved representation of $\cA$ is induced from a one-dimensional representation. Thereby, the induction procedure is the generalized Rieffel induction introduced and studied in \cite{SS}. This result can be viewed as an analogue of the following theorem by Kirillov (see \cite{Kir}): Every  irreducible unitary representation of a nilpotent Lie group is induced from a one-dimensional representation of a certain subgroup.

In each of three cases the constructed crossed product $C^*$-algebra is of a special kind. Namely, the partial action of $\dZ$ on $C_0(\cBp)$ is generated by a single partial automorphism $\Theta,$ see \cite{Ex,McC} and Section \ref{subsec_part_act}. In this case the partial crossed product $C^*$-algebra coincides with the covariance $C^*$-algebra of the partial automorphism in the sense of \cite[Definition 3.7]{Ex}. In the case $\cA=\cO(S_{qr}^2)$ all $*$-representations are bounded, hence well-behaved, and the crossed product $C^*$-algebra $C_0(\cBp)\times_\alpha\dZ$ is isomorphic to the enveloping $C^*$-algebra of $C_{env}^*(\cA).$ 

Finally, for $\cA_q$ and for $\cU_q(su(2))$ we show the existence of ``bad'' representations. More precisely, we prove the existence of a $*$-representation which is not well-behaved and cannot be extended to a well-behaved $*$-representation even in a larger Hilbert space. It generalizes the well-known results for $W_1$ and $\cU(su(2)).$

Among the examples which can be analyzed in the same spirit include various bounded and unbounded $*$-algebras: quantum group algebras $SU_q(2),\ SU_q(1,1),$ $q$-deformed $\cU(su(1,1)),$ different deformations of CAR and CCR, AF pre-$C^*$-algebras (see \cite{Ex1}) etc.

\subsection{$*$-Algebras and $*$-representations} By a $*$-\textit{algebra} we mean a complex associative algebra $\cA$ equipped with a mapping $a\mapsto a^*$ of $\cA$ into itself, called the \textit{involution} of $\cA$, such that $(\lambda a+\mu b)^* = \bar{\lambda}a^*+ \bar{\mu} b^*, (ab)^* = b^* a^*$ and $(a^*)^*=a$ for $a,b\in \cA$ and $\lambda, \mu\in \dC$. The unit of $\cA$ (if it exists) will be denoted by $\Un{\cA}$ or simply by $\Un{}.$ For every $*$-algebra $\cA$ denote by $\sum\cA^2$ the set of finite sums $\sum a_i^*a_i^{},\ a_i\in\cA.$

Throughout this paper we use some terminology and results from unbounded representation theory in Hilbert space (see e.g. \cite{s}). We repeat some basic notions and facts. If $T$ is a Hilbert space operator, $\cD(T),\ \overline{T}$ and $T^*$ denote its domain, its closure and its adjoint, respectively. Let $\cD$ be a dense linear subspace of a Hilbert space $\cH$ with scalar product $\langle\cdot,\cdot\rangle.$ A $*$-\textit{representation} of a $*$-algebra $\cA$ on $\cD$ is an algebra homomorphism $\pi$ of $\cA$ into the algebra $L(\cD)$ of linear operators on $\cD$ such that $\langle\pi(a)\varphi,\psi\rangle=\langle\varphi,\pi(a^*)\psi\rangle$ for all $\varphi,\psi\in\cD$ and $a\in \cA$. We call $\cD(\pi):=\cD$ the \textit{domain} of $\pi$ and write $\cH_\pi:=\cH$. Two $*$-representations $\pi_1$ and $\pi_2$ of $\cA$ are \textit{(unitarily) equivalent} if there exists an isometric linear mapping $U$ of $\cD(\pi_1)$ onto $\cD(\pi_2)$ such that $\pi_2(a)=U\pi_1(a)U^{-1}$ for $a\in \cA$. The \textit{direct sum representation} $\pi_1\oplus\pi_2$ acts on the domain $\cD(\pi_1)\oplus \cD(\pi_2)$ by $(\pi_1\oplus \pi_2)(a)=\pi_1(a)\oplus \pi_2(a)$, $a\in\cA$. A $*$-representation $\pi$ is called \textit{irreducible} if a direct sum decomposition $\pi=\pi_1\oplus\pi_2$ is only possible when $\cD(\pi_1)=\{0\}$ or $\cD(\pi_2)=\{0\}$. For a $*$-subalgebra $\cB\subseteq\cA$ we denote by $\Res_\cB\pi$ its restriction to $\cB.$ The \textit{graph topology} of $\pi$ is the locally convex topology on the vector space $\cD(\pi)$ defined by the norms $\varphi\mapsto\norm{\varphi}+\norm{\pi(a)\varphi},$ where $a\in\cA$. If $\cD(\overline{\pi})$ denotes the completion the $\cD(\pi)$ in the graph topology of $\pi,$ then $\overline{\pi}(a):=\overline{\pi(a)}\upharpoonright \cD(\overline{\pi}),$ $a\in\cA,$ defines a $*$-representation of $\cA$ with domain $\cD(\overline{\pi}),$ called the \textit{closure} of $\pi.$ In particular, $\pi$ is \textit{closed} if and only if $\cD(\pi)$ is complete in the graph topology of $\pi.$ A $*$-representation $\pi$ is called \textit{non-degenerate} if $\pi(\cA)\cD(\pi):= {\rm Lin}~\{\pi(a)\varphi; a \in \cA, \varphi \in \cD(\pi)\}$ is dense in $\cD(\pi)$ in the graph topology of $\pi$. If $\cA$ is unital and $\pi$ is non-degenerate, then we have $\pi(\Un{\cA})\varphi=\varphi$ for all $\varphi \in \cD(\pi)$. We say that $\pi$ is \textit{cyclic} if there exists a vector $\varphi\in \cD(\pi)$ such that $\pi(\cA)\varphi$ is dense in $\cD(\pi)$ in the graph topology of $\pi$. For a $C^*$-algebra $\gA$ and Hilbert space $\cH,$ denote by $\Rep(\gA,\cH)$ the category of non-degenerate $*$-representations of $\gA$ on $\cH.$ By $\Rep\gA$ denote the category of all non-degenerate $*$-representations of $\gA.$

We recall the induction procedure for $*$-representations of general $*$-algebras developed in \cite[Section 2]{SS} in a slightly more general context. However, we will not perform this procedure but use Proposition \ref{orthonormal base of representation space} to get the explicit formulas. Let $\cB\subseteq\cA$ be $*$-algebras. A linear map $p:\cA\to\cB$ is called a \textit{bimodule projection} if $p(a^*)=p(a)^*,\ p(b_1ab_2)=b_1p(a)b_2,\ p(\Un{\cA})=\Un{\cB},$ for all $a\in\cA,\ b_1,b_2\in\cB.$ Let $\rho$ be a $*$-representation of $\cB.$ Denote by $\cA\otimes_\cB\cD(\rho)$ the quotient of $\cA\otimes_\dC\cB$ by the linear span of vectors $ab\otimes\varphi-a\otimes\rho(b)\varphi,a\in\cA,b\in\cB,\varphi\in\cD(\rho).$ We say that $\rho$ is \textit{ inducible} from $\cB$ to $\cA$ via $p$ if the sesquilinear form
\begin{gather}\label{eq_sesq}
\langle\sum_k x_k\otimes\varphi_k,\sum_l y_l\otimes\psi_l\rangle_0:=\sum_{k,l}\langle\rho(p(y_l^*x_k^{}))\varphi_k,\psi_l\rangle,
\end{gather}
is positive semi-definite on $\cA\otimes_\cB\cD(\rho).$ Denote by $\cK_\rho$ the kernel of $\langle\cdot,\cdot\rangle_0.$ Then $\cD_0=\cA\otimes_\cB\cD(\rho)/\cK_\rho$ is an inner-product space. Define a $*$-representation $\pi$ on $\cD_0$ via
$$
\pi(a)(\sum_i[a_i\otimes\varphi_i]):=\sum_i[aa_i\otimes\varphi_i],
$$
where $\sum_i[a_i\otimes\varphi_i]\in\cD_0$ denotes the image of $\sum_i a_i\otimes\varphi_i$ under the quotient mapping. Finally define $\Ind\rho$ to be the closure of $\pi.$

Our major application of the induction procedure will be in the following context. Let $G$ be a discrete group and $\cA$ be a \textit{$G$-graded $*$-algebra}. That is, $\cA$ is a direct sum of vector spaces $\cA_g,\ g \in G$, such that
\begin{gather}\label{eq_Ggrad}
\cA_g\cdot\cA_h\subseteq \cA_{g\cdot h}\ \mbox{and}\
(\cA_g)^*\subseteq\cA_{g^{-1}}\ \mbox{for}\ g,h\in G.
\end{gather}
The elements of $\cup_{g\in G}\cA_g$ are called \textit{homogeneous}. For every subgroup $H\subseteq G$ the sum $\oplus_{g\in H}\cA_g=:\cA_H$ is a $*$-subalgebra of $\cA$ and the canonical projection $p:\cA\to\cA_H$ is a bimodule projection. If $\cA_e$ is commutative then a character $\chi:\cA_e\to\dC$ is inducible (via $p_e:\cA\to\cA_e$) if and only if $\chi(a^*a)\geq 0$ for all homogeneous $a\in\cA.$

\subsection{Partial actions and partial crossed products}\label{subsec_part_act} The constructions and results of this subsection are taken from \cite{Ex,McC}. A \textit{partial action} of a discrete group $G$ on a set $X$ is a pair
$$\alpha=(\set{\cD_g}_{g\in G},\set{\alpha_g}_{g\in G}),$$
where $\cD_g\subseteq X,\ g\in G$ are subsets and $\alpha_g:\cD_{g^{-1}}\to \cD_g$ are bijections such that
\begin{enumerate}
  \item[(i)] $\alpha_g(\cD_{g^{-1}}\cap\cD_h)=\cD_{gh}\cap\cD_g,\ g,h\in G,$
  \item[(ii)] $\alpha_{hg}(x)=\alpha_h(\alpha_g(x)),\ x\in\cD_{g^{-1}}\cap\cD_{g^{-1}h^{-1}},$
  \item[(iii)] $\cD_e=X,\ \alpha_e=\Id_X.$
\end{enumerate}

For a partial action $\alpha=(\set{\cD_g}_{g\in G},\set{\alpha_g}_{g\in G})$ on a topological space $X$ we require in addition that $\cD_g$ are open sets and $\alpha_g:\cD_{g^{-1}}\to\cD_g,\ g\in G$ are homeomorphisms. We call $(X,G,\alpha)$ a \textit{partial dynamical system (p.d.s.)}. 

For a partial action $\beta=(\set{I_g}_{g\in G},\set{\beta_g}_{g\in G})$ of $G$ on a $C^*$-algebra $\gB$ we require in addition that $I_g,\ g\in G$ are closed two-sided ideals and $\beta_g:I_{g^{-1}}\to I_g$ are $*$-isomorphisms. We call $(\gB,G,\beta)$ a \textit{partial $C^*$-dynamical system ($C^*$-p.d.s.)}. For a p.d.s. $(X,G,\alpha)$ where $X$ is a locally compact Hausdorff space we define the \textit{dual} $C^*$-p.d.s. as follows. Put $\gB=C_0(X),\ I_g=C_0(\cD_g)$ and define $\beta_g:I_{g^{-1}}\to I_g$ by
$$
(\beta_g(f))(x)=f(\alpha_{g^{-1}}(x)),\ x\in\cD_{g}, f\in I_{g^{-1}},\ g\in G.
$$
Direct computations show that $\beta=(\set{I_g}_{g\in G},\set{\beta_g}_{g\in G})$ is a partial action on $\gB$ and that $(\gB,G,\beta)$ is a $C^*$-p.d.s.

Let $(\gB,G,\beta),\ \beta=(\set{I_g}_{g\in G},\set{\beta_g}_{g\in G})$ be a $C^*$-p.d.s. The \textit{partial crossed product $C^*$-algebra} $\gA=\gB\times_\beta G$ is the enveloping $C^*$-algebra of the $*$-algebra $\gB G$ defined as follows. $\gB G\subseteq\gB\otimes\dC[G]$ is the linear span of the set $\set{a\otimes g\mid a\in\cI_g},$ with multiplication and involution defined by
$$
(a\otimes g)(b\otimes h):=\alpha_{g}(\alpha_{g^{-1}}(a)b)\otimes gh,\ (a\otimes g)^*:=\alpha_{g^{-1}}(a^*)\otimes g^{-1}.
$$

The examples of $C^*$-p.d.s. which appear below are of a special kind. Recall \cite{Ex}, that a \textit{partial automorphism} of a $C^*$-algebra $\gB$ is a triple $\Theta=(\theta,I,J),$ where $I,J\subseteq\gB$ are closed two-sided ideals and $\theta:I\to J,$ is a $*$-isomorphism. Set $I_0=\gB$ and define $I_n,\ n\in\dZ,$ by induction
\begin{gather*}
I_{n+1}=\set{a\in J\mid \theta^{-1}(a)\in I_n},\ \mbox{for}\ n\geq 0,\\
I_{n-1}=\set{a\in I\mid\theta(a)\in I_n},\ \mbox{for}\ n\leq 0.
\end{gather*}
In particular, $I=I_{-1}$ and $J=I_1.$ It can be checked, see \cite[Section 3]{Ex}, that the triple $(\gB,\dZ,\beta),$ where $\beta=(\set{I_n}_{n\in\dZ},\set{\theta^n}_{n\in\dZ})$ is a $C^*$-p.d.s. The partial crossed product algebra $\gB\times_\beta\dZ$ is called the \textit{covariance algebra of} $(\gB,\Theta)$ and is denoted by $C^*(\gB,\Theta).$ As in the case of a crossed-product by a $*$-automorphism, $*$-representations of $C^*(\gB,\Theta)$ are in one-to-one correspondence with covariant representations of the pair $(\gB,\Theta),$ see \cite[Section 5]{Ex}. In case of the $C^*$-p.d.s. defined by $(\gB,\Theta)$ a \textit{covariant representation} $\pi\times u$ consists of a $*$-representation $\pi:\gB\to B(\cH)$ and a partial isometry $u,$ whose initial and final spaces are $\overline{\pi(I)\cH}$ and $\overline{\pi(J)\cH}$ respectively, so that
$$
\pi(\theta(b))=u\pi(b)u^*,\ \mbox{holds for every}\ b\in I.
$$
If the latter is satisfied, then $\pi\times u$ becomes a $*$-representation of $\gB\dZ,$ hence of $C^*(\gB,\dZ),$ via
$$
(\pi\times u)(f\otimes k)=\pi(f)u^k,\ \mbox{for}\ f\otimes k\in\gB\dZ,
$$
where $u^{-k}=u^{*k}$ for $k\in\dN.$

\subsection{Unbounded elements affiliated with $C^*$-algebras and $C^*$-envelopes} The theory of unbounded elements affiliated with a $C^*$-algebra was developed in \cite{wor1}, see also \cite{Lan}. Let $\gA$ be a $C^*$-algebra and let $T$ be a densely defined closed linear operator on $\gA.$ Denote by $D(T)\subseteq\gA$ its domain\footnote{Recall that $\cD(\cdot)$ is domain of a Hilbert space operator.}. The adjoint operator $T^*$ is defined as follows. For $y,z\in\gA$ write $y\in D(T^*),\ T^*y=z$ if $\langle Tx,y\rangle=\langle x,z\rangle$ holds for all $x\in D(T).$ Following \cite{Lan} we say that $T$ is \textit{affiliated\footnote{In \cite{Lan} the term \textit{regular operator} on $\gA$ is used.} with $\gA$} and write $T\,\eta\,\gA,$ if $D(T^*)$ and the range of $1+T^*T$ are dense in $\gA$, see \cite[Chapter 9]{Lan}. 

Every non-degenerate $*$-representation of a $C^*$-algebra $\gA$ can be continued to the set $\gA^\eta$ of all operators affiliated with $\gA.$ Namely, for every $\pi\in\Rep(\gA,\cH)$ and $T\,\eta\,\gA,$ there exists a closed operator $\pi(T)\,\eta\,\pi(\gA)$ with a core $\pi(D(T))\cH$ such that 
$$
\pi(T)(\pi(a)\varphi)=\pi(Ta)\varphi,\ \mbox{for all}\ \varphi\in\cH,\ a\in D(T).
$$
Moreover, if $D_0\subseteq D(T)$ is a core of $T,$ then $\pi(D_0)\cH$ is a core of $\pi(T).$

\begin{defn}\label{defn_envelope}
Let $\cA$ be a $*$-algebra with a given category of $*$-representations $\Rep\cA$ and fixed generators $a_1,\dots,a_n.$ We will say that a $C^*$-algebra $\gA$ is a \textit{$C^*$-envelope} of $\cA$ if there exist affiliated elements $A_1,\dots,A_n\,\eta\,\gA$ such that 
\begin{gather}\label{eq_correspondence_A_A}
\pi(A_i)=\overline{\rho(a_i)},\ i=1,\dots,n.
\end{gather}
defines an equivalence functor $\rho\mapsto\pi$ between $\Rep\cA$ and $\Rep\gA.$
\end{defn}
\noindent \textit{Remarks} 1. If every $*$-representation of $\cA$ is bounded, then there exists the enveloping $C^*$-algebra $C_{env}^*(\cA),$ which is obviously a $C^*$-envelope of $\cA.$ 

\noindent 2. In the last definition, the isomorphism class of $\gA$ depends a priori on the choice of the generators $a_i$ and of the category $\Rep\cA.$ However, we cannot provide any example, where $\gA$ would depend on the generators $a_i.$

\section{The orbit method}

In this section we recall the orbit method developed in \cite{SS}. Throughout the section $G$ is a countable discrete group and $\cA$ is a \textit{$G$-graded $*$-algebra}. We assume that the $*$-subalgebra $\cB:=\cA_e$ is commutative and denote by $\cBd$ the set of all characters of $\mathcal{B}$ (i.e. nontrivial $\ast$-homomorphisms $\chi:\mathcal{B}\rightarrow\mathbb{C}$). Further, we define the ''positive'' spectrum $\widehat{\mathcal{B}}^+\subseteq\cBd$ to be the set of all characters $\chi\in\cBd$ which satisfy\footnote{The theory developed in \cite{SS} requires the additional condition $\chi(c^{\ast}d)\chi(d^{\ast}c)=\chi(c^{\ast}c)\chi(d^{\ast}d)\textnormal{ for all }\chi\in\cBp,g\in G,c,d\in\mathcal{A}_g,$ which holds automatically. It can be checked using the equation $(c^*cd^*d)^2=(c^*cd^*d)(c^*dd^*c)$ which follows by commutativity of $\cB.$}
\begin{gather}
\chi(a^*a)\geq 0\ \mbox{for all homogeneous elements}\ a\in\cA. 
\end{gather}
\begin{lem}\label{lem_poschar}
Assume that for every $g\in G$ there exists an element $a_g\in\mathcal{A}_g$ such that $\mathcal{A}_g=a_g\mathcal{B}$. Then $\chi\in\cBd$ belongs to $\cBp$ if and only if $\chi(a_g^*a_g^{})\geq 0$ for all $g\in G$.
\end{lem}
\begin{prf}The ``only if'' part is clear. Assume that $\chi(a_g^*a_g^{})\geq 0$ for all $g\in G$. By assumption, if $c_g\in\mathcal{A}_g$, then $c_g=a_gb$ for some $b\in\mathcal{B}$. Hence 
	\begin{gather*}
		\chi(c_g^*c_g)=\chi(b^*a_g^*a_g^{}b)=\chi(a_g^*a_g^{})\chi(b^*b)\geq 0.
	\end{gather*}
\end{prf}

The set $\cBp$ consists of those characters which satisfy \eqref{eq_sesq}, i.e. are inducible from $\cB$ to $\cA$ via $p_{e}.$ 

\begin{defn}
For $g\in G$ define\footnote{In \cite{SS} the notation $\alpha_g:\cD_g^{}\to\cD_{g^{-1}}$ was used.}
\begin{gather}\label{eq_defn_cDg}
\cD_{g^{-1}}=\set{\chi\in\cBp\mid \chi(a_g^*a_g^{})\neq 0\ \mbox{for some}\ a_g\in\cA_g}.
\end{gather}
If $\chi\in\cD_{g^{-1}}$ and $\chi(a_g^*a_g^{})\neq 0$ we set
\begin{gather}\label{character defined}
	(\alpha_g(\chi))(b):=\frac{\chi(a_g^{\ast}ba_g)}{\chi(a_g^{\ast}a_g)} \mbox{ for }b\in\mathcal{B}.
\end{gather}
\end{defn}

Direct computations (see \cite[Proposition 13]{SS}) show that $\alpha=(\set{\alpha_g}_{g\in G},\set{\cD_g}_{g\in G})$ is a well-defined partial action of $G$ on $\cBp.$ We will often write $\chi^g$ instead of $\alpha_g(\chi).$ For a character $\chi\in\cBp$ we denote by $\Orb\chi\subseteq\cBp$ its orbit under the partial action of $G.$

\begin{prop}[see Proposition 16 in \cite{SS}]\label{orthonormal base of representation space}
Let $\chi\in\widehat{\mathcal{B}}^+$ and $\pi=\Ind\chi$ be the induced $*$-representation. For every $g\in G$ such that $\chi\in\cD_{g^{-1}}$ fix an element $a_g\in\cA_g$ such that $\chi(a_g^*a_g^{})\neq 0.$ Then there exists an orthonormal base $\set{e_g\mid\chi\in\cD_{g^{-1}}}$ in $\cD(\pi)$ such that for $h\in G$ and $b_h\in\mathcal{A}_h$ we have
\begin{align*}
\pi(b_h)e_g=\frac{\chi(a_{hg}^{\ast}b_ha_g)}{\chi(a_{hg}^{\ast}a_{hg}^{})^{1/2}\chi(a_g^{\ast}a_g^{})^{1/2}}e_{hg}, \mbox{ if } \chi\in\cD_{g^{-1}h^{-1}}
\end{align*}
and $\pi(b_h)e_g=0$ otherwise. In particular, if $b\in\mathcal{B}$, we have $\pi(b)e_g=\chi^g(b)e_g.$
\end{prop}

For an element $b\in\cB$ introduce its \textit{"Gel'fand transform"}
$$
\widehat{b}:\cBd\to\dC,\ \widehat{b}(\chi)=\chi(b),\ \chi\in\cBd.
$$
We equip $\cBd$ with the weak topology defined by $\set{\widehat{b}\mid b\in\cB}$ and the Borel structure generated by the open sets. By definition of $\cBp$ it is a closed subset of $\cBd.$ It can be checked, that the partial action of $G$ is topological. That is, $\cD_g,\ g\in G$ are open sets, and $\alpha_g:\cD_{g^{-1}}\to\cD_g$ are homeomorphisms. Since $G$ is countable and the one-point sets are closed, the $G$-orbits are Borel subsets of $\cBp.$

\begin{defn}\label{definition well-behaved representation}
	A closed $\ast$-representation $\pi$ of $\mathcal{A}$ is called \textit{well-behaved} if:
	\begin{itemize}
		\item [$(i)$] the restriction $\textnormal{Res}_{\mathcal{B}}\pi$ of $\pi$ to $\mathcal{B}$ is integrable and there exists a spectral measure $E_{\pi}$ on $\cBp$ such that
		\begin{align*}
			\overline{\pi(b)}=\int_{\widehat{\mathcal{B}}^+}{\widehat{b}(\chi)dE_{\pi}(\chi)} \textnormal{ for } b\in\mathcal{B}.
		\end{align*}
		\item [$(ii)$] For all $a_g\in\mathcal{A}_g,g \in G$, and all Borel subsets $\Delta\subseteq\widehat{\mathcal{B}}^+$, we have
		\begin{gather}\label{eq_defn_well_beh}
			\pi(a_g)E_{\pi}(\Delta)\supseteq E_{\pi}(\alpha_g(\Delta\cap\cD_{g^{-1}}))\pi(a_g).
		\end{gather}
	\end{itemize}
A well-behaved $*$-representation $\pi$ is \textit{associated with an orbit} $\Orb\chi$ if $E_{\pi}$ is supported on the set $\Orb\chi.$ Denote by $\Rep\cA$ the category of well-behaved $*$-representations.
\end{defn}

\medskip
By \cite[Proposition 17]{SS}, relation \eqref{eq_defn_well_beh} can be replaced with
\begin{gather}\label{eq_defn_well_beh_u}
u_g\int f(t)dE_\pi(t)\subseteq\int_{\cD_g}f(\alpha_{g^{-1}}(t))dE_\pi(t)\cdot u_g.
\end{gather}
where $u_g$ is the partial isometry in the polar decomposition $\overline{\pi(a_g)}=u_gc_g,$ and $f$ is any measurable function on $\cBp.$ If $f$ is bounded, then ``$\subseteq$'' becomes an equality.

In the next proposition we collect basic properties of well-behaved $*$-representations. For the proof see Propositions 18, 29 and Theorem 7 in \cite{SS}.
\begin{prop}\label{prop_properties}
	\begin{itemize}
		\item [(i)] Every bounded $*$-representation is well-behaved.
		\item [(ii)] If the partial action of $G$ on $\cBp$ posesses a measurable countably separated section, then every irreducible well-behaved $*$-representation is associated with an orbit.
		\item [(iii)] Condition $(i)$ in Definition \ref{definition well-behaved representation} holds automatically if $\cB$ is countably generated, and the restriction of $\pi$ on $\cB$ is integrable, that is $\overline{\pi(a)}$ is normal for all $b\in\cB.$
	\end{itemize}
	\label{prop_well}
\end{prop}

A measurable set $\Gamma$ is \textit{countably separated} if and only if there exist Borel sets $B_k,\ k\in\mathbb{N},\ \Gamma\subseteq\bigcup_{k\in\mathbb{N}}B_k$ such that for arbitrary $x,y\in\Gamma,\ x\neq y,$ we have $x\in B_{k_0},\ y\notin B_{k_0}$ for some $k_0\in\mathbb{N}$. A subset $\Gamma$ containing exactly one point from each orbit is called a \textit{section} of a partial dynamical system. 

Recall, that for a subgroup $H\subseteq G,$ $\cA_H=\oplus_{g\in H}\cA_g$ is a $*$-subalgebra of $\cA.$

\begin{thm}[See Theorem 5 in \cite{SS}]
	Let $\chi\in\widehat{\mathcal{B}}^+$ be a character and let $H=\St\chi$ be its stabilizer group. Then the map
	\begin{align*}
		\rho\mapsto\textnormal{Ind}_{\mathcal{A}_H\uparrow\mathcal{A}}(\rho)=\pi
	\end{align*}
	is a bijection from the set of unitary equivalence classes of inducible \sloppy$\ast$-representations $\rho$ of $\mathcal{A}_H$ for which
	\begin{align}\label{Res multiple character}
		\textnormal{Res}_{\mathcal{B}}\rho \textit{ corresponds to a multiple of the character }\chi
	\end{align}
	onto the set of unitary equivalence classes of well-behaved \sloppy$\ast$-representations $\pi$ of $\mathcal{A}$ associated with $\Orb\chi$. A \sloppy$\ast$-representation $\rho$ satisfying \eqref{Res multiple character} is bounded and inducible. Moreover, $\pi$ is irreducible if and only if $\rho$ is irreducible.
	\label{thm_all_induced}
\end{thm}
The last theorem suggests the following algorithm for description of all irreducible well-behaved $*$-representations of $\cA:$
\begin{itemize}
 \item determine $\cBd,\ \cBp,$ the partial action of $G$ on $\cBp$ and a section $\Gamma\subseteq\cBp$,
 \item for each $\chi\in\Gamma$
  \begin{itemize}
   \item if the stabilizer $\St\chi$ is trivial, compute $\Ind\chi,$
   \item otherwise find all irreducible representations $\rho$ of $\cA_{\St\chi}$ satisfying \eqref{Res multiple character} and compute $\Ind\rho.$
  \end{itemize}
\end{itemize}

If the Proposition \ref{prop_properties}, (ii) applies, then we obtain all irreducible well-behaved representations of $\cA.$

\section{The $q$-Oscillator Algebra}
\label{sec: q-oscillator algebra}

By the \textit{quantum harmonic oscillator} ($q$-\textit{oscillator}) we mean the following relation
\begin{gather}\label{eq_qosc}
aa^{\ast}=\Un{}+qa^{\ast}a,\ q>0.
\end{gather}

In this section, we use the notation of the $q$-calculus $[[k]]_q=1+q+\dots+q^{k-1}.$ Further, we put 
$$
F(t):=1+qt.
$$ 
Clearly $F([[k]]_q)=[[k+1]]_q,\ k\in\dN_0.$

In \cite{CGP} the authors have obtained the following representations of \eqref{eq_qosc} by Hilbert space operators:
\begin{compactitem}
\item For every $q>0$ the Fock representation $\pi_F$ acting on the orthonormal base $\set{\gee_k}_{k\in\dN_0}$ as
	\begin{align}\label{eq_fock_rep}
		\pi_F(a)\gee_k=[[k]]_q^{1/2}\gee_{k-1},\ 
		\pi_F(a^{\ast})\gee_k=[[k+1]]_q^{1/2}\gee_{k+1},\ \mbox{where}\ \gee_{-1}:=0.
	\end{align}
\item  For $q\in(0,1)$ the series of unbounded $*$-representations $\pi_\gamma, \gamma\in(0,1]$ acting on the orthonormal base $\set{\gee_k}_{k\in\dZ}$ as
\begin{align}\label{eq_gamma_rep}
	\pi_\gamma(a)\gee_k=\left(\frac{1+q^{\gamma +k}}{1-q}\right)^{1/2}\gee_{k+1},\ 
	\pi_\gamma(a^*)\gee_k=\left(\frac{1+q^{\gamma +k+1}}{1-q}\right)^{1/2}\gee_{k-1}.
\end{align}
\item  For $q\in(0,1)$ the series of one-dimensional $*$-representations
\begin{gather}\label{eq_onedim_rep}
	\pi_\varphi(a)=e^{\img\varphi}(1-q)^{-1/2},\; \pi_\varphi(a^{\ast})=e^{-\img\varphi}(1-q)^{-1/2},\ \varphi\in [0,2\pi).
\end{gather}
\end{compactitem}

Using the orbit method described in the previous section, we classify all irreducible well-behaved $*$-representations of the \textit{$q$-oscillator algebra}
\begin{align*}
	\mathcal{A}=\mathbb{C}\langle a,a^{\ast}\mid aa^{\ast}=qa^{\ast}a+\Un{}\rangle,\qquad q>0.
\end{align*}
We will see that the formulas for the irreducible well-behaved $*$-representations of $\cA$ coincide with \eqref{eq_fock_rep}--\eqref{eq_onedim_rep}.

\medskip
We now introduce the ingredients needed for the orbit method. Define the $\dZ$-grading on $\cA$ by setting $a\in\cA_1,\; a^{\ast}\in\cA_{-1}$ and put $\cB:=\cA_0.$ It is easily checked that $\cB=\dC[N]$, where $N=a^*a,$ and $\cA_n=a^n\cB,\ \cA_{-n}=a^{*n}\cB$ for every $n\in\dN.$ Using induction on $k\in\dN$ we obtain the relations
\begin{align}
\begin{split}
a^ka^{\ast k}&=\prod_{j=1}^k{\left(q^jN+[[j]]_q\Un{}\right)},\ k\in\dN.\\
a^{\ast k}a^k&=\prod_{j=0}^{k-1}{\left( q^{-j}N+[[-j]]_q\Un{}\right)},\ k\in\dN.
\label{q-oscillator algebra relations}
\end{split}
\end{align}

Since $\cB=\dC[N],$ every character on $\cB$ is of the form $\chi_t(N)=t\in\mathbb{R}.$ In what follows we identify the space of all characters $\cBd$ with $\dR.$ 

\begin{prop}\label{prop_q_osc_cBp}
	\begin{itemize}
		\item [(i)] $\begin{aligned}[t]
		 \cBp &=\left\{[[k]]_q\mid k\in\mathbb{N}_0\right\}\textnormal{ for } q\geq 1,\\
		 \cBp &=\left\{[[k]]_q\mid k\in\mathbb{N}_0\right\}\cup[1/(1-q),+\infty)\textnormal{ for } q\in(0,1).
		\end{aligned}$
		\item [(ii)] The partial action $\alpha=\left(\{\mathcal{D}_{n}\}_{n\in\mathbb{Z}},\{\alpha_n\}_{n\in\mathbb{Z}}\right)$ is given as follows. 
		\begin{align*}
			\mathcal{D}_{-n} &=\left\{ [[k]]_q\mid k\geq n \right\}\textnormal{ if } q\geq 1,\\
			\mathcal{D}_{-n} &=\left\{ [[k]]_q\mid k\geq n \right\}\cup\left[1/(1-q),\infty\right) \textnormal{ if } q\in(0,1).
		\end{align*}If $\chi_t\in\mathcal{D}_{-n}$, then $\chi_t^n=\chi_{F^{-n}(t)}.$ In particular, $\chi_{[[k]]_q}^n=\chi_{[[k-n]]_q}^{}$ for $n\leq k$. 
	\end{itemize}
\end{prop}
\begin{proof}
$(i)$ By Lemma \ref{lem_poschar} a character $\chi\in\cBd$ belongs to $\cBp$ if and only if $\chi(a^ka^{\ast k})\geq 0,\ \chi(a^{\ast k}a^k)\geq 0$ for all $k\in\dN.$ Further, \eqref{q-oscillator algebra relations} implies that $a^na^{*n}=\sum_{j=0}^n\alpha_jN^j$ for some $\alpha_j\geq 0.$ Hence $\chi_t\in\widehat{\mathcal{B}}^+$ if and only if $\chi(a^{\ast k}a^k)=\chi\left(\prod_{j=0}^{k-1} {\left(q^{-j}N+[[-j]]_q\Un{}\right)}\right)\geq  0$ for all $k\in\mathbb{N}$. The last system of inequalities is equivalent to
\begin{align}
	\prod_{j=0}^{k}{(t-[[j]]_q)}\geq 0 \textnormal{ for all } k\in\mathbb{N}_0.
	\label{eq_cBp}
\end{align}
Consider first $q\geq 1$. Then $[[k]]_q\rightarrow\infty,\ k\rightarrow\infty,$ and \eqref{eq_cBp} is satisfied if and only if $t=[[k]]_q$ for some $k\in\mathbb{N}_0$. If $q\in(0,1),$ then $[[k]]_q\rightarrow\frac{1}{1-q},\ k\rightarrow\infty,$ and every $t\geq \frac{1}{1-q}$ satisfies \eqref{eq_cBp}. For $t\in\left[0,\frac{1}{1-q}\right),$ \eqref{eq_cBp} holds if and only if $t=[[k]]_q$ for some $k\in\mathbb{N}_0$.\\	
$(ii)$ One can verify by induction on $n\in\mathbb{N}$ that $F^{n}(t)=q^nt+[[n]]_q$ for all $n\in\dZ$. Using \eqref{q-oscillator algebra relations}, we obtain 
$$
\chi_t^n(N)=\frac{\chi_t(a^{*n}Na^n)}{\chi_t(a^{*n}a^n)}=\chi_t(q^{-n}N+[[-n]]_q)=\chi_{F^{-n}(t)}(N),
$$
for $\chi_t\in\mathcal{D}_{-n},\ n\in\dN$, and
\begin{align*}
	\chi_t^{-n}(N)=\frac{\chi_t(a^nNa^{*n})}{\chi_t(a^na^{*n})}=q^{-1}\chi_t(q^{n+1}N+[[n+1]]_q)-q^{-1}=\chi_{F^n(t)}(N).
\end{align*}
for $\chi_t\in\cD_n,\ n\in\dN.$ Inequalities \eqref{eq_cBp} imply that for $q>0$ and $t=[[k]]_q$, we have $\chi_t\in\mathcal{D}_{-n}$ if and only if $n\leq k$. In case $q\in(0,1)$ and $t\geq \frac{1}{1-q}$, we have $\chi_t\in\mathcal{D}_{-n}$ for all $n\in\dZ$.
\end{proof}

Using Proposition \ref{prop_q_osc_cBp}, we conclude that the stabilizer $\St\chi_t$ of $\chi_t\in\cBp$ is trivial except for the case $t=1/(1-q)$, where the stabilizer is $\dZ.$ Define the subset $\Gamma\subseteq\cBp$ as 
\begin{align*}
	\Gamma&=\set{0}\cup\set{\frac{1}{1-q}}\cup\set{\frac{1+q^{\gamma}}{1-q} \mid \gamma\in (0,1]},\;\textnormal{ if } q\in(0,1),\\
	\Gamma&=\set{0},\;\textnormal{ if } q\geq 1.
\end{align*}
Direct computations using Proposition \ref{prop_q_osc_cBp} show that each orbit under the partial action of $\mathbb{Z}$ on $\widehat{\mathcal{B}}^+$ intersects $\Gamma$ in exactly one point, i.e. $\Gamma$ is a section of the partial action. The topology on $\widehat{\mathcal{B}}^+$ is induced from the standard topology on $\mathbb{R}$. Hence $\Gamma$ is countably separated and measurable. By Proposition \ref{prop_well}(ii) every irreducible well-behaved $*$-representation of $\mathcal{A}$ is associated to some $\textnormal{Orb}\chi,\ \chi\in\Gamma$. For $\chi\in\Gamma$ we consider three cases.
\begin{itemize}
	\item [$(i)$] Case $\chi=\chi_0.$ Since the stabilizer of $\chi$ is trivial, the only irreducible well-behaved $*$-representation associated to $\textnormal{Orb}\chi$ is (up to unitary equivalence) $\pi_F:=\textnormal{Ind}\chi$. Using Proposition \ref{orthonormal base of representation space} and relations \eqref{q-oscillator algebra relations}, we calculate the action of $\pi_F$ on the orthonormal basis $\{e_{-n}\}_{n\in\mathbb{N}_0}$ of the representation space $\mathcal{H}_{\pi_F}$.
	\begin{align*}
	\pi_F(a)e_{-n}&=\frac{\chi(a^{n-1}aa^{*n})}{\chi(a^{n-1}a^{*(n-1)})^{1/2}\chi(a^na^{*n})^{1/2}}e_{-n+1}=\frac{\chi(a^{n}a^{*n})^{1/2}}{\chi(a^{n-1}a^{*(n-1)})^{1/2}}e_{-n+1}\\
	       &=q^{n/2}\left(\chi(N) -[[-n]]_q \right)^{1/2}e_{-n+1}=[[n]]_q^{1/2}e_{-n+1},\\
	\pi_F(a^{\ast})e_{-n}&=\frac{\chi_0(a^{n+1}a^{\ast}a^{\ast n})}{\chi_0(a^{n+1}a^{*(n+1)})^{1/2}\chi_0(a^na^{\ast n})^{1/2}}e_{-n-1}=\frac{\chi_0(a^{n+1}a^{*(n+1)})^{1/2}}{\chi_0(a^na^{\ast n})^{1/2}}e_{-n-1}\\
			     &=q^{(n+1)/2}(\chi(N)-[[-n-1]]_q)^{1/2}e_{-n-1}=[[n+1]]_q^{1/2}e_{-n-1},
	\end{align*}
	where $e_1:=0$ and $n\in\dN_0$. It exists for any $q>0$ and is bounded if and only if $q\in (0,1)$.	
	
	\item [$(ii)$] Case $\chi=\chi_{\frac{1+q^{\gamma}}{1-q}},\ \gamma\in(0,1].$ The stabilizer of $\chi$ is again trivial, thus $\pi_\gamma:=\textnormal{Ind}\chi_{\frac{1+q^{\gamma}}{1-q}}$ is the only irreducible well-behaved $*$-representation associated to $\textnormal{Orb}\chi$. We calculate the action of $\pi_{\gamma}(a)$ respectively $\pi_{\gamma}(a^{\ast})$ using Proposition \ref{orthonormal base of representation space} and relations \eqref{q-oscillator algebra relations}. For $n\in\dZ$ we have
	\begin{align*}\pi_{\gamma}(a)e_n&=\frac{\chi(a^{*(n+1)}aa^n)}{\chi(a^{\ast(n+1)}a^{n+1})^{1/2}\chi(a^{\ast n}a^n)^{1/2}}e_{n+1}=\frac{\chi(a^{\ast(n+1)}a^{n+1})^{1/2}}{\chi(a^{\ast n}a^n)^{1/2}}e_{n+1}\\
	&=\left(q^{-n}\frac{1+q^{\gamma}}{1-q}+\frac{1-q^{-n}}{1-q}\right)^{1/2}e_{n+1}=\left(\frac{1+q^{\gamma-n}}{1-q}\right)^{1/2}e_{n+1}.
	\end{align*}
	In the same way we obtain
	\begin{align*}
	\pi_\gamma(a^*)e_n=\left(\frac{1+q^{\gamma-n+1}}{1-q}\right)^{1/2}e_{n-1},\ \mbox{for}\ n\in\mathbb{Z}.
	\end{align*}
	Note that $\pi_{\gamma}$ is not bounded for every $\gamma\in(0,1]$.
	
	\item [$(iii)$] Case $\chi=\chi_{\frac{1}{1-q}}.$ The stabilizer group $H$ of $\chi_{\frac{1}{1-q}}$ is $\mathbb{Z}.$ Let $\rho$ be an irreducible $*$-representation of $\cA$ satisfying \eqref{Res multiple character}. Since $\chi(aa^{\ast}-a^{\ast}a)=0$, we have $\rho(a)\rho(a^{\ast})=\rho(a^{\ast})\rho(a).$ By Schur's Lemma $\rho$ is one-dimensional. For $\lambda\in\dC$ such that $\lambda=\rho(a)$ we get $\left|\lambda\right|^2 =\rho(aa^{\ast})=\rho(\Un{}+qa^{\ast}a)=1+q\left|\lambda\right|^2.$  Hence $\rho=\rho_\varphi,$ for some $\varphi\in[0,2\pi),$ where $\rho_\varphi(a)=e^{\textnormal{i}\varphi}(1-q)^{-1/2}.$ Since $\cA_H=\cA,\ \pi_\varphi:=\Ind_{\cA_H\uparrow\mathcal{A}}\rho_\varphi$ is equivalent to $\rho_\varphi$ and we have
	\begin{align*}
	\pi_\varphi(a)=e^{\img\varphi}(1-q)^{-1/2},\; \pi_\varphi(a^{\ast})=e^{-\img\varphi}(1-q)^{-1/2},\ \varphi\in [0,2\pi).
	\end{align*}
\end{itemize}
	
By Theorem \ref{thm_all_induced} these are all (up to unitary equivalence) irreducible well-behaved $\ast$-representations of $\cA.$ Moreover, putting $\gee_k:=e_{-k},\ k\in\mathbb{Z},$ we see that the above formulas coincide with \eqref{eq_fock_rep}, \eqref{eq_gamma_rep} and \eqref{eq_onedim_rep} respectively. We have proved the following

\begin{thm}
Every irreducible well-behaved $*$-representation of the $q$-oscillator algebra, $q>0$, is induced from a one-dimensional $*$-representation.
\end{thm}

\subsection{Existence of bad $*$-representations}
In this subsection we prove the existence of a $*$-representation $\pi$ of $\cA$ which is not well-behaved and which cannot be continued to a well-behaved representation in a possibly larger Hilbert space. The idea is similar to the proof of \cite[Theorem 4.1]{s3}.

\begin{lem}
The polynomial 
$$
p:=(N-1)(N-(1+q))\in\mathbb{C}[N]
$$
is positive in every well-behaved $*$-representation of $\cA$ and $p\notin\sum\cA^2.$
\label{lem_pospol_qccr}
\end{lem}

\begin{proof}
We first show that every element of $\sum{\cA^2\cap\cB}$ is of the form 
\begin{gather}\label{eq_aux3}
\sum_{k=0}^na^{*k}a^k\cdot p_k^*p_k^{},\ \mbox{where}\ p_k\in\dC[N],\ n\in\dN.
\end{gather}
Indeed, an element $b\in\mathcal{B}$ belongs to $\sum\cA^2$ if and only if $b=\sum{b_j^*b_j}$, where $b_j\in\mathcal{A}_{k_j}$, $k_j\in\mathbb{Z}$. Since $\mathcal{A}_k=a^k\cdot\mathcal{B}$ for every $k\in\mathbb{Z}$ (here $a^{-k}=a^{*k}$ for $k>0$), we obtain $b=\sum{a^{*k}a^k\cdot s_k}$, where $s_k\in\sum{\mathcal{B}^2}$. It is a well-known fact, that every positive polynomial in $\mathbb{C}[N]$ is a single square $r^*r$. Hence $s_k=p_k^*p_k^{},\ p_k\in\mathbb{C}[N],$ for $k\in\mathbb{Z}$. Furthermore, relations \eqref{q-oscillator algebra relations} imply $a^na^{*n}\in\sum{\mathcal{B}^2}+a^*a\sum{\mathcal{B}^2},$ which proves \eqref{eq_aux3}.

Let $\pi$ be a well-behaved $*$-representation of $\cA$ with associated spectral measure $E_{\pi}$. Since $\supp E_{\pi}\subseteq\cBp$ and $p\geq 0$ on $\cBp$, we have $\pi(p)=\int_{\cBp}p(\lambda)dE_{\pi}(\lambda)\geq 0$.

Assume to the contrary that $p\in\sum{\cA^2}$. Since the degree of $p(N)$ in $\dC[N]$ is $2$, we get by \eqref{eq_aux3}
$$
p=f^{\ast}f+a^*a\cdot g^{\ast}g+a^{*2}a^2\cdot h^{\ast}h=f^{\ast}f+Ng^{\ast}g+N(N-1)h^{\ast}h
$$ 
for some polynomials $f,g,h\in\mathbb{C}[N],$ where $\deg f\leq 1,\ \deg g=0$ and $\deg h=0$, that is, $g$ and $h$ are constant. Setting $N:=1$, we obtain $\abs{f(1)}^2+\abs{g}^2=0$, i.e. $g=0,\ f(1)=0$. Setting $N:=1+q$, we get $\abs{f(1+q)}^2+q(1+q)\abs{h}^2=0$ which implies $h=f(1+q)=0$. Since $\deg f\leq 1,\ f\equiv 0$, i.e. $p\equiv 0$, a contradiction.
\end{proof}

For the prove of the next theorem we will need the following technical result, see \cite[Lemma 2]{s2}.
\begin{lem}\label{closcone1}
Let $\cA$ be a unital $\ast$-algebra which has a faithful $\ast$-representation $\pi$ (that is, $\pi(a)=0$ implies that $a=0$) and is a union of a sequence of finite dimensional subspaces $E_n$, $n \in \dN$. Assume that for each $n \in \dN$ there exists a number $k_n\in \dN$ such that the following is satisfied: If $a \in \sum \cA^2$ is in $E_n$, then we can write $a$ as a finite sum $\sum a_j^*a_j^{}$ such that all $a_j$ are in $E_{k_n}$.\\
Then the cone $\sum \cA^2$ is closed in $\cA$ with respect to the finest locally convex topology on $\cA$.
\end{lem}

\begin{thm}\label{thm_badrep_qccr}
There exists a $*$-representation $\pi$ of the $q$-oscillator algebra $\cA$ which cannot be extended to a well-behaved representation in a possibly larger Hilbert space.
\end{thm}

\begin{proof}
Since $p\notin\sum{\mathcal{A}^2}$ and $\sum{\mathcal{A}^2}$ is closed by Lemma \ref{closcone1}, there exists a linear functional $\varphi:\mathcal{A}\rightarrow\mathbb{C}$ such that $\varphi(\sum{\mathcal{A}^2})\geq 0$ and $\varphi(p)<0$ by the Hahn-Banach Theorem. Let $(\pi_{\varphi},\mathcal{H}_{\varphi},\xi_{\varphi})$ be its GNS-construction (see \cite[Section 8.6.]{s}). Assume to the contrary, that that $\pi_\varphi$ has a well-behaved extension, say $\pi.$ Then $\langle \pi(p)\xi_{\varphi},\xi_{\varphi}\rangle=\langle \pi_{\varphi}(p)\xi_{\varphi},\xi_{\varphi}\rangle=\varphi(p)<0$. On the other hand, $\pi(p)\geq 0$ by Lemma \ref{lem_pospol_qccr}, a contradiction.
\end{proof}

\subsection{$C^*$-envelope of the $q$-oscillator algebra}
In this subsection we show that $\cA,$ considered with the category $\Rep\cA$ and generators $a,a^*$ has a $C^*$-envelope $\gA$ in the sense of Definition \ref{defn_envelope}. For let $(C_0(\cBp),\dZ,\beta)$ be the $C^*$-p.d.s. dual to $(\cBp,\dZ,\alpha).$ More precisely, define the partial action $\beta=(\set{I_k}_{k\in\dZ},\set{\beta_k}_{k\in\dZ})$ on $C_0(\cBp)$ by setting $I_k:=C_0(\cD_k)$ and
$$
(\beta_k(f))(t):=f(\alpha_{-k}(t))=f(F^k(t)),\ \mbox{for}\ f\in I_{-k},\ t\in\cD_k.
$$
Proposition \ref{prop_q_osc_cBp} implies that the $C^*$-p.d.s. $(C_0(\cBp),\dZ,\beta)$ is defined by the partial automorphism $\Theta=(\theta,I,J),$ where
$$
I=I_{-1},\ J=I_1=C_0(\cBp),\ (\theta(f))(t)=(\beta_1(f))(t)=f(1+qt),\ f\in I_{-1}.
$$
We define 
$$
\gA:=C^*(C_0(\cBp),\Theta)=C_0(\cBp)\times_\beta\dZ.
$$ 

\begin{thm}\label{thm_envelope_qccr}
Consider the $q$-oscillator algebra $\cA$ with generators $a,a^*$ and the category of well-behaved representations $\Rep\cA.$ Then $\gA$ is a $C^*$-envelope of $\cA.$
\end{thm}
\begin{proof}
Let $\gA_0$ be the linear hull of 
$$
\set{f\otimes k\in\gA\mid k\in\dZ,\ \supp f\subseteq\cD_k\ \mbox{is compact}}.
$$ 
$\gA_0$ is obviously dense in $\gA.$ For $f\otimes k\in\gA_0$ define
\begin{gather*}
A(f(t)\otimes k)=\sqrt{1+qt}f(1+qt)\otimes (k+1),
\end{gather*}
and
\begin{gather*}
A^*(f(t)\otimes k)=\sqrt t f(q^{-1}t-q^{-1})\otimes (k-1)
\end{gather*}
Then $A$ and $A^*$ are densely defined linear operators on $\gA$ and their closures, denoted again by $A$ and $A^*,$ are adjoint to each other. For $f\otimes k\in\gA_0$ we have
\begin{gather*}
A^*A(f(t)\otimes k)=A^*(\sqrt{\alpha_{-1}(t)}f(\alpha_{-1}(t))\otimes (k+1))=tf(t)\otimes k.
\end{gather*}

The last equation shows that the range of $I+A^*A$ is dense in $\gA_0\subseteq\gA,$ so that $A$ is affiliated with $\gA.$ By \cite[Theorem 1.4.]{wor1} the adjoint $A^*$ is also affiliated with $\gA.$ We show the correspondence \eqref{eq_correspondence_A_A} between the generators $a,a^*\in\cA$ and affiliated elements $A,A^*\,\eta\,\gA.$

By \cite[Theorem 5.6]{Ex} every $*$-representation of $\gA$ is given by a covariant representation $\pi\times u$ of $(C_0(\cBp),\Theta).$ Here $\pi:C_0(\cBp)\to B(\cH_\pi)$ is a $*$-representation of $C_0(\cBp)$ and $u$ is a partial isometry on $\cH_\pi$ satisfying $\pi(\theta(b))=u\pi(b)u^*$ for every $b\in I.$ By the spectral theory of commutative $C^*$-algebras, there exists a unique spectral measure $E_\pi$ on $\cBp$ such that
$$
\pi(f)=\int_\cBp f(t)dE_\pi(t),\ f\in C_0(\cBp).
$$
By definition of $\theta$ for $f\in I_{-1}=C_0(\cD_{-1})$ we have 
$$
u\left(\int fdE_\pi\right)u^*=u\pi(f)u^*=\pi(\theta(f))=\int f(1+qt)dE_\pi(t).
$$
Multiplying the latter by $u$ from the right and remembering that the initial space of $u$ is $\overline{\pi(I_{-1})\cH_\pi}=E_\pi(\cD_{-1})\cH_\pi$ we get
\begin{gather}\label{aux1}
u\int fdE_\pi=\int f(1+qt)dE_\pi(t)\cdot u,\ \mbox{for}\ f\in I_{-1}.
\end{gather}
The extension of $(\pi\times u)$ to $A$ and $A^*$ is given by
\begin{gather}\label{eq_aux4}
(\pi\times u)(A)=u\int\sqrt{t}dE_\pi,\ (\pi\times u)(A^*)=\int\sqrt tdE_\pi(t)\cdot u^*
\end{gather}
Indeed, for every $f\otimes k\in\gA_0$ we have
\begin{gather*}
u\int\sqrt{t}dE_\pi\cdot ((\pi\times u)(f\otimes k))=u\int\sqrt{t}dE_\pi\cdot \int fdE_\pi\cdot u^k=\\
=u\int\sqrt{t}f(t)dE_\pi\cdot u^k=\int\sqrt{1+qt}f(1+qt)dE_\pi\cdot u^{k+1}=(\pi\times u)(A(f\otimes k)).
\end{gather*}
Since $\gA_0\subseteq\gA$ is a core of $A,\ \pi(\gA_0)\cH_\pi$ is a core of $A,$ and we get the first part of \eqref{eq_aux4}. The second part follows from $(\pi\times u)(A^*)=((\pi\times u)(A))^*.$

Let $\rho$ be a well-behaved $*$-representation of $\cA,$ and $E_\rho$ be the corresponding spectral measure on $\cBp\subseteq\dR_+.$ Further, let $\overline{\rho(a)}=u_1 c_1$ be the polar decomposition of $\overline{\rho(a)}.$ Since $a^*a$ is the generator of $\cB,\ E_\rho$ coincides with the spectral measure of $\overline{\rho(a^*a)}=\rho(a)^*\overline{\rho(a)}.$ Hence 
\begin{gather}\label{eq_aux6}
\overline{\rho(a)}=u_1\int\sqrt tdE_\rho,\ \mbox{and}\ \overline{\rho(a^*)}=\int\sqrt tdE_\rho\cdot u_1^*.
\end{gather}
Since $\ker u_1=\ker c_1,$ the initial space of $u_1$ is the range of $E_\rho(\cBp\setminus\set{0})=E_\rho(\cD_{-1}).$ Further, we have $u_1^{}c_1^2u_1^*=1+qc_1^2,$ which implies that $\ker u_1^*$ is trivial, so that the final space of $u_1$ is $E_\rho(\cBp)=E_\rho(\cD_1).$ Applying \eqref{eq_defn_well_beh_u} to $f\in C_0(\cD_{-1})$ we obtain
\begin{gather*}
u_1\rho(f)u_1^*=u_1^{}\int f(t)dE_\rho(t)\cdot u_1^*=\int f(\alpha_{-1}(t))dE_\rho(t)\cdot u_1^*u_1^{}=\\
=\int f(1+qt)dE_\rho(t)=\rho(\theta(f)),
\end{gather*}
i.e. $(\rho_\cB\times u_1),$ where $\rho_\cB$ is the restriction of $\rho$ to $\cB,$ defines a covariant representation of $(C_0(\cBp),\Theta).$ 

The correspondence \eqref{eq_correspondence_A_A} between $\pi$ and $\rho$ follows now by comparing \eqref{eq_aux4} with \eqref{eq_aux6}.
\end{proof}

\noindent \textit{Remark.} In \cite[Section 3]{wor2} the author shows that the operators $p,q$ of the Weyl algebra $W_1$ generate a $C^*$-algebra $\gA$ in the sense of the Definition 3.1 therein, and that $\gA$ is the algebra of compact operators. It corresponds to the fact that the $C^*$-envelope of $q$-CCR with $q=1$ is isomorphic to the partial crossed product $C_0(\dN_0)\times_\alpha\dZ\simeq K(l^2(N_0)).$

\section{The Podle\'s Sphere}\label{sec: Podles sphere}
In this section we investigate $*$-representations of the Podles' sphere $\mathcal{O}(S_{qr}^2).$ We consider only the case $q\in(0,1),\ r\in(0,\infty).$ The cases $r=0,\ r=\infty$ can be treated similarly. Recall \cite{Pd} that $\cA:=\mathcal{O}(S_{qr}^2)$ is the unital $\ast$-algebra generated by $a=a^{\ast},b,b^{\ast}$ and defining relations
\begin{align}
	ab=q^{-2}ba,\; ab^{\ast}=q^2b^{\ast}a,\; b^{\ast}b=a-a^2+r\Un{},\; bb^{\ast}=q^2a-q^4a^2+r\Un{}.
	\label{eq_podles}
\end{align}
The defining relations imply that every $*$-representation of $\mathcal{A}$ is bounded and hence well-behaved by Proposition \ref{prop_properties} (i). In \cite{Pd} the following irreducible $*$-representations of $\cA$ were obtained.

\begin{compactitem}
 \item  Two infinite-dimensional $*$-representations $\pi_{\pm}$ which act on an orthonormal base $\{\gee_k\}_{k\in\mathbb{N}_0}$ of the representation space $\mathcal{H}_{\pm}$ by
\begin{align*}
	\pi_{\pm}(a)\gee_k=q^{2k}\lambda_{\pm}\gee_k,\ \pi_{\pm}(b)\gee_k=\left(q^{2k}\lambda_{\pm}-(q^{2k}\lambda_{\pm})^2+r\right)^{1/2}\gee_{k-1},\\
	\pi_{\pm}(b^*)\gee_k=\left(q^{2(k+1)}\lambda_{\pm}-(q^{2(k+1)}\lambda_{\pm})^2+r\right)^{1/2}\gee_{k+1},\ \gee_{-1}:=0,
\end{align*}
where $\lambda_{\pm}:=\frac{1}{2}\pm(r+\frac{1}{4})^{1/2}$.
\item The series of one-dimensional $\ast$-representations $\pi_{\varphi},\ \varphi\in[0,2\pi)$,
\begin{align*}
	\pi_{\varphi}(a)=0,\ \pi_{\varphi}(b)=e^{\textnormal{i}\varphi}r^{1/2},\ \pi_{\varphi}(b^{\ast})=e^{-\textnormal{i}\varphi}r^{1/2}.
\end{align*}
\end{compactitem}
Using \eqref{eq_podles} and induction on $n\in\dN$ we obtain the following relations
\begin{align}
\begin{split}
	ab^n&=q^{-2n}b^na,\
	ab^{\ast n}=q^{2n}b^{\ast n}a.\\
	b^{\ast n}b^n&=\prod_{j=1}^{n}{\left(q^{-2(j-1)}a-q^{-4(j-1)}a^2+r\right)},\
	b^nb^{\ast n}=\prod_{j=1}^n{\left(q^{2j}a-q^{4j}a^2+r\right)}.
	\label{eq_podles_pol_rel}
\end{split}
\end{align}
Define a $\dZ$-grading on $\mathcal{A}$ by setting $a\in\cA_0,\; b\in\mathcal{A}_1$ and $b^{\ast}\in\mathcal{A}_{-1}$. Using the defining relations one easily derives
\begin{align*}
	\cB:=\cA_0=\textnormal{Lin}\{a^lb^{\ast m}b^m\mid l,m\in\mathbb{N}_0\},\quad
	\mathcal{A}_n:=b^n\mathcal{B},\quad
	\mathcal{A}_{-n}:=b^{\ast n}\mathcal{B},
	\end{align*}
where $n\in\mathbb{N}_0$. Further, relations \eqref{eq_podles_pol_rel} imply that $\cB=\dC[a],$ hence $\cBd=\set{\chi_t\mid t\in\dR}$ where $\chi_t(a)=t.$ As in the previous section we identify $\cBd$ with $\mathbb{R}$.

\begin{prop}\label{prop_cBp_alpha_podles}
	\begin{itemize}
		\item [(i)] $\cBp=\set{\chi_{m,+}}_{m\in\dN_0}\cup\set{\chi_{m,-}}_{m\in\dN_0}\cup\set{\chi_\infty},$\\
where $\chi_{m,\pm}$ denotes $\chi_t,\ t=q^{2m}\lambda_{\pm}$ and $\chi_\infty$ denotes $\chi_t,\ t=0.$
		\item [(ii)] 	The partial action $\alpha=\left( \{\mathcal{D}_n\}_{n\in\mathbb{Z}},\{\alpha_n\}_{n\in\mathbb{Z}} \right)$ is given as follows:
	\begin{gather*}
		\mathcal{D}_{-n}=\set{\chi_{m,\pm}\mid m\geq n}\cup\set{\chi_\infty}, \mbox{ and } \chi_{m,\pm}^n=\chi_{m-n,\pm},\ \chi_\infty^n=\chi_\infty.
	\end{gather*}
	\end{itemize}
\end{prop}
\begin{proof} $(i)$ Lemma \ref{lem_poschar} and relations \eqref{eq_podles_pol_rel} imply that $\chi_t,\ t\in\dR$ belongs to $\cBp$ if and only if the following inequalities are satisfied for all $n\in\mathbb{N}$:
	\begin{align}\label{Podles sphere: 1}
	\begin{split}
		\chi_t(b^{\ast n}b^n)&=\prod_{k=0}^{n-1}(q^{-2k}t-q^{-4k}t^2+r)\geq 0,\\
		\chi_t(b^nb^{\ast n})&=\prod_{k=1}^n(q^{2k}t-q^{4k}t^2+r)\geq 0.
	\end{split}
	\end{align}
	Assume $q^{-2k}t-q^{-4k}t^2+r>0$ for all $k\in\dN_0.$ Since $q^{-2k}\rightarrow+\infty,\ k\to+\infty,$ it is possible only if $t=0,$ i.e. $\chi_t=\chi_\infty.$ If $t\neq 0$ we get $q^{-2k}t-q^{-4k}t^2+r=0$ for some $k\in\dN_0,$ whence
	\begin{align*}
		t=\frac{-q^{-2k}\pm\sqrt{q^{-4k}+4q^{-4k}r}}{-2q^{-4k}}=q^{2k}\lambda_{\mp}.
	\end{align*}
  One can easily check that every $t=q^{2k}\lambda_{\pm},\ k\in\dN_0,$ satisfies \eqref{Podles sphere: 1}.

$(ii)$ 	Relations \eqref{Podles sphere: 1} imply that $\mathcal{D}_{-n}=\set{\chi_{m,\pm}\mid m\geq n}\cup\set{\chi_\infty}$. Assume that $\chi_{m,\pm}\in\mathcal{D}_{-n}$, where $n\in\mathbb{N}_0$. Using relations \eqref{eq_podles_pol_rel} we obtain
	\begin{align*}
		\chi_{m,\pm}^n(a)&=\frac{\chi_{m,\pm}(b^{\ast n}ab^n)}{\chi_{m,\pm}(b^{\ast n}b^n)}
		=\frac{\chi_{m,\pm}(b^{\ast n}b^n)\chi_{m,\pm}(q^{-2n}a)}{\chi_{m,\pm}(b^{\ast n}b^n)}
		=\chi_{m-n,\pm}(a).
	\end{align*}
	For $\chi_\infty$ we have $\chi_\infty(b^{\ast n}b^n)=\chi_\infty(b^nb^{\ast n})=r^n\neq 0$ for all $n\in\mathbb{Z}$ by equations \eqref{Podles sphere: 1}. Hence $\chi_\infty\in\cD_n$ for all $n\in\mathbb{Z}$ and $\chi_\infty^n(a)=0$.
\end{proof}

Let $\Gamma$ be the subset $\{\chi_{0,+},\chi_{0,-},\chi_\infty\}\subseteq\cBp.$ Obviously $\Gamma$ is a measurable countably separated section of the p. d. s. $(\cBp,\dZ,\alpha).$ We calculate all irreducible \sloppy $*$-re\-pre\-sen\-ta\-ti\-ons associated with $\Orb\chi,\ \chi\in\Gamma.$

\begin{enumerate}
 \item [$(i)$] Case $\chi_{0,\pm}$. The stabilizer of $\chi_{0,\pm}$ is trivial by Proposition \ref{prop_cBp_alpha_podles}, (ii). Put $\pi_\pm:=\Ind\chi_{0,\pm}$. We use Proposition \ref{orthonormal base of representation space} to compute the action of $\pi_\pm$ on the orthonormal base $\{e_{-k}\}_{k\in\dN_0}$.
\begin{align*}
	\pi_\pm(b)e_{-k}&=\left(q^{2k}\chi_{0,\pm}\left(a \right)-q^{4k}\chi_{0,\pm}\left( a^2 \right)+r\right)^{1/2}e_{-k+1}\\
	&=\left(q^{2k}\lambda_{\pm} - \left(q^{2k}\lambda_{\pm} \right)^2+r\right)^{1/2}e_{-k+1},\\
	\pi_\pm(b^{\ast})e_{-k}&=\left(q^{2(k+1)}\lambda_{\pm} - \left(q^{2(k+1)}\lambda_{\pm}\right)^2+r\right)^{1/2}e_{-k-1},\\
	\pi_\pm(a)e_{-k}&=\chi_{0,\pm}^{-k}(a)=q^{2k}\lambda_{\pm}e_{-k}.
\end{align*}
\item [$(ii)$] Case $\chi_\infty.$ The stabilizer group $H$ of $\chi$ is $\mathbb{Z}.$ Let $\rho$ be an irreducible $*$-representation of $\cA_{H}$ satisfying \eqref{Res multiple character}. Since $\chi(bb^{\ast}-b^{\ast}b)=0$, we have $\rho(b)\rho(b^{\ast})=\rho(b^{\ast})\rho(b).$ By Schur's Lemma $\rho$ is one-dimensional. For $\lambda\in\dC$ such that $\lambda=\rho(b)$ we get $\left|\lambda\right|^2 =\rho(bb^{\ast})=r.$  Hence $\rho=\rho_\varphi,$ for some $\varphi\in[0,2\pi),$ where $\rho_\varphi(b)=e^{\textnormal{i}\varphi}r^{1/2}.$ Since $\cA_H=\cA,\ \pi_\varphi:=\Ind_{\cA_H\uparrow\mathcal{A}}\rho_\varphi$ is equivalent to $\rho_\varphi$ and we get
\begin{align*}
\pi_{\varphi}(a)=0,\ \pi_\varphi(b)=e^{\img\varphi}r^{1/2},\; \pi_\varphi(b^{\ast})=e^{-\img\varphi}r^{1/2},\ \varphi\in [0,2\pi).
\end{align*}

\end{enumerate}

By Theorem \ref{thm_all_induced} these are all, up to unitary equivalence, irreducible $\ast$-representations of $\cA.$ Setting $\gee_k:=e_{-k},\ k\in\mathbb{N}_0,$ we see that these coincide with the ones found in \cite{Pd}. In particular, we have the following

\begin{thm}
Every irreducible $*$-representation of the Podle\'s sphere $\mathcal{O}(S_{qr}^2),\ q\in(0,1),\ r\in(0,\infty)$ is induced from a one-dimensional $*$-representation.
\end{thm}

In the remaining part of this section we describe the enveloping $C^*$-algebra of $\mathcal{O}(S_{qr}^2).$ For let  $(C_0(\cBp),\dZ,\beta)$ be the $C^*$-p.d.s. dual to $(\cBp,\dZ,\alpha)$ as defined in the Subsection \ref{subsec_part_act}. Note, that the sets $\cD_k,\ k\in\dZ,$ are compact, hence $I_k:=C(D_k).$ By definition of $\beta$ we have
\begin{gather}\label{eq_beta_Podles}
(\beta_k(f))(t)=f(\alpha_{-k}(t))=f(q^{2k}t),\ f\in I_{-k},k\in\dZ.
\end{gather}
It is easily seen from the description of $\alpha$ in the Proposition \ref{prop_cBp_alpha_podles} that the partial action $\beta=(\set{I_k}_{k\in\dZ},\set{\beta_k}_{k\in\dZ})$ is defined by the partial automorphism $\Theta=(\theta,I,J),$ where $\theta=\beta_1,\ I=I_{-1},\ J=I_1=\cA.$ 

\begin{thm}
The enveloping $C^*$-algebra $\gA$ of $\cA$ is isomorphic to the covariance algebra $C^*(C(\cBp),\Theta)\simeq C(\cBp)\times_\beta\dZ$ 
\end{thm}

\begin{proof}
The proof goes similarly to the proof of the Theorem \ref{thm_envelope_qccr} by replacing the $\eta$-relation with $\in$-relation. 

We first define $*$-homomorphism $\epsilon:\cA\to C(\cBp)\times_\beta\dZ$ by setting
\begin{gather*}
\epsilon(a)=t\otimes 0,\ \epsilon(b)=(q^2t-q^4t^2+r)^{1/2}\otimes 1,\ \epsilon(b^*)=(t-t^2+r)^{1/2}\otimes (-1).
\end{gather*}
Direct computations using \eqref{eq_beta_Podles} show that $\epsilon(a),\epsilon(b),\epsilon(b^*)$ satisfy the defining relations of $\mathcal{O}(S_{qr}^2),$ that is, $\epsilon$ is well-defined. Every representation $\pi$ of $\cA$ is bounded, hence well-behaved by Proposition \ref{prop_properties}. That is, $\pi$ gives rise to a covariant representation $\pi|_\cB\times u,$ where $u$ is the partial isometry in the polar decomposition $\pi(b)=uc.$ On the other hand, every $*$-representation of $\gA$ is given by a covariant representation of the partial automorphism $\Theta.$ This proves the correspondence \eqref{eq_correspondence_A_A} for the representations of $\cA$ and $\gA.$
\end{proof}

\section{The Quantum Algebra $\cU_q(su(2))$}

In this section $\cA$ is the $q$-deformed enveloping $*$-algebra $\cU_q(su(2)),\ q>0,\ q\neq 1,$ which is generated by $E,F,K,K^{-1}$ satisfying the following defining relations
\begin{gather*}
	KK^{-1}=K^{-1}K=\Un{},\quad KEK^{-1}=q^2E,\quad KFK^{-1}=q^{-2}F,\\
	[E,F]=EF-FE=\frac{K-K^{-1}}{q-q^{-1}},\\
	E^{\ast}=FK,\quad F^{\ast}=K^{-1}E,\quad K^{\ast}=K.
\end{gather*}
In this section, we use the standard notation $[n]\equiv[n]_q=\frac{q^n-q^{-n}}{q-q^{-1}}$, where $n\in\dZ$ and $q\neq 0$. Further $X^0$ denotes $\Un{}$ if $X$ is one of the four generators $E,F,K,K^{-1}$. 

In \cite{vaksman} the authors considered the family of irreducible $*$-representations $\set{\pi_{\omega,l}\mid\omega=\pm 1,\ l\in\frac{1}{2}\mathbb{N}_0}$ of $\mathcal{U}_q(su(2)).$ The representation $\pi_{\omega,l}$ acts on an orthonormal base $\{\gee_m\}_{m=-l,\ldots,l}$ of the representation space as follows:
	\begin{align}	
	\begin{split}
		\pi_{\omega,l}(K)\gee_m &=\omega q^{2m}\gee_m,\\
		\pi_{\omega,l}(E)\gee_m &=q^{m+1}\sqrt{[l-m][l+m+1]}\gee_{m+1},\\
		\pi_{\omega,l}(F)\gee_m &=\omega q^{-m}\sqrt{[l+m][l-m+1]}\gee_{m-1}.
		\label{eq_formulas}
	\end{split}
	\end{align}
We will show that every irreducible well-behaved representation of $\cA$ is unitarily equivalent to $\pi_{\omega,l}$ for some $l\in\frac{1}{2}\dN_0,\ \omega=\pm 1.$ 

\medskip
Define a $\mathbb{Z}$-grading of $\cA$ by setting $E\in\mathcal{A}_1,\; F\in\mathcal{A}_{-1}$ and $K,K^{-1}\in\mathcal{A}_0$. Then
\begin{align*}
	\mathcal{B}:=\mathcal{A}_0=\textnormal{Lin}\{F^lK^mE^l\mid l\in\mathbb{N}_0,\;m\in\mathbb{Z}\}=\textnormal{Lin}\{E^lK^mF^l\mid l\in\mathbb{N}_0,\;m\in\mathbb{Z}\}.
\end{align*}
The $\ast$-subalgebra $\mathcal{B}\subseteq\cA$ is commutative and is equal to $\mathbb{C}[EF,K,K^{-1}]=\mathbb{C}[C_q,K,K^{-1}]$. For $n\in\mathbb{N}_0$, we have
\begin{align*}
	\mathcal{A}_n=E^n\mathcal{B}=\textnormal{Lin}\{E^{n+l}K^mF^l\mid l\in\mathbb{N}_0,\;m\in\mathbb{Z}\},\\
	\mathcal{A}_{-n}=F^n\mathcal{B}=\textnormal{Lin}\{F^{n+l}K^mE^l\mid l\in\mathbb{N}_0,\;m\in\mathbb{Z}\}.
\end{align*}
One can verify by a direct computation that the \textit{quantum Casimir element} $C_q$ is a central element in $\mathcal{A}$, where
\begin{align*}
	C_q=EF+\frac{q^{-1}K+qK^{-1}}{(q-q^{-1})^2}.
\end{align*}

The following lemma can be easily proved by induction.
\begin{lem}
For every $n\in\dN$ we have
\begin{align*}
		(i)\ [E,F^n]&\equiv EF^n-F^nE=[n]F^{n-1}[K;1-n],
		\\
		(ii)\ [E^n,F]&\equiv E^nF-FE^n=[n]E^{n-1}[K;n-1],
\end{align*}
where we set $[K;l]:=(q^lK-q^{-l}K^{-1})/(q-q^{-1})$ for $l\in\mathbb{Z}$.
	\label{quantum algebra: relations}
\end{lem}
This lemma implies the following relations:
	\begin{align}
		\begin{split}
		E^nF^n&=\prod_{j=1}^n{(EF+[j-1][K;-j])},\\
		F^nE^n&=\prod_{j=1}^n{(EF-[j][K;j-1])},\ n\in\mathbb{N}.
		\end{split}
		\label{quantum algebra: polynomial relations}
	\end{align}

Since $\cB=\mathbb{C}[C_q,K,K^{-1}],$ every character $\chi\in\cBd$ is equal to some $\chi_{st}\in\widehat{\mathcal{B}},\;(s,t)\in\mathbb{R}\times\mathbb{R}\backslash\{0\}$ where
	\begin{align*}
		\chi_{st}(C_q)=s,\quad\chi_{st}(K)=t.
	\end{align*}
\begin{prop}
	\begin{itemize}
	\item[(i)] A character $\chi_{st}\in\widehat{\mathcal{B}}$ belongs to $\widehat{\mathcal{B}}^+$ if and only if
	\begin{align*}
		t=\pm q^{m-n} \mbox{ and }s=\frac{\pm q^{m+n+1}\pm q^{-m-n-1}}{(q-q^{-1})^2}, \textit{ where }m,n\in\mathbb{N}_0.
	\end{align*}
In particular, 
$$
\cBp=\set{\chi_{m,n,+}\mid m,n\in\dN_0}\cup\set{\chi_{m,n,-}\mid m,n\in\dN_0},
$$
where $\chi_{m,n,\pm}=\chi_{st}$ with $s,t$ from above.
	\item[(ii)] 	The partial action $\alpha=\left( \{\mathcal{D}_{n}\}_{n\in\dZ},\{\alpha_n\}_{n\in\dZ} \right)$ is given as follows:
\begin{gather*}
 \mathcal{D}_{-k}=\set{\chi_{m,n,\pm}\mid -m\leq k\leq n},\ \mbox{and}\ \chi_{m,n,\pm}^k=\chi_{m+k,n-k,\pm}^{}.
\end{gather*}
	\end{itemize}
	\label{quantum algebra: set of positive characters}
\end{prop}
\begin{proof}
	$(i)$ Lemma \ref{lem_poschar} and	Equations \eqref{quantum algebra: polynomial relations} imply that $\chi\in\widehat{\mathcal{B}}^+$ if and only if the following inequalities are satisfied for arbitrary $k\in\mathbb{N}$:
	\begin{align}
		\chi(E^kF^kK^{-k})&=\prod_{j=1}^{k}{\chi(EF+[j-1][K;-j])}\chi(K)^{-k}\geq 0.
		\label{quantum algebra: inequality 1}\\
		\chi(F^kE^kK^k)&=\prod_{j=1}^{k}{\chi(EF-[j][K;j-1])}\chi(K)^k\geq 0,
		\label{quantum algebra: inequality 2}
	\end{align}
	We show that there exist $m,n\in\mathbb{N}_0$ such that
	\begin{align}
		\chi(EF+[m][K;-m-1])&=0,
		\label{quantum algebra: equation 1}\\
		\chi(EF-[n+1][K;n])&=0.
		\label{quantum algebra: equation 2}
	\end{align}
	Assume the contrary, i.e. $\chi(EF+[k][K;-k-1])\neq 0$ for all $k\in\mathbb{N}_0$. Suppose $t>0$. Then, by \eqref{quantum algebra: inequality 1},
	\begin{align*}
		\chi(EF)\geq -[k]\chi([K;-k-1])=\frac{(q^{-2k-1}-q^{-1})t+(q^{2k+1}-q)t^{-1}}{(q-q^{-1})^2}.
	\end{align*}
	Such an value $\chi(EF)\in\mathbb{R}$ cannot exist, since $q^{-2k-1}\rightarrow \infty$ for $k\rightarrow \infty$ if $q\in (0,1)$, respectively $q^{2k+1}\rightarrow \infty$ for $k\rightarrow \infty$ if $q>1$. Analogously one obtains a contradiction for $t<0$, using inequalities \eqref{quantum algebra: inequality 1}. Thus, $\chi(EF+[m][K;-m-1])=0$ for some $m\in\mathbb{N}_0$. Similarly one can prove that $\chi(EF-[n+1][K;n])=0$ for some $n\in\mathbb{N}_0$, using inequalities \eqref{quantum algebra: inequality 2}. Subtracting \eqref{quantum algebra: equation 2} from \eqref{quantum algebra: equation 1} yields
	\begin{eqnarray*}
		& &[m]\chi([K;-m-1])=-[n+1]\chi([K;n])\\
		&\Longleftrightarrow &(q^{-1}-q^{-2m-1})t-(q^{2m+1}-q)t^{-1}=\\
		& &=(q^{-1}-q^{2n+1})t-(q^{-2n-1}-q)t^{-1}\\
		&\Longleftrightarrow &t^2=\frac{q^{2m+1}-q^{-2n-1}}{q^{2n+1}-q^{-2m-1}}=\frac{q^{2m}(q-q^{-2m-2n-1})}{q^{2n}(q-q^{-2m-2n-1})}\\
		&\Longleftrightarrow &t=\pm q^{m-n}.
	\end{eqnarray*}
	Hence, we obtain
	\begin{align*}
		\chi(C_q)&=[n+1]\chi([K;n])+\frac{q^{-1}\chi(K)+q\chi(K^{-1})}{(q-q^{-1})^2}\\
		&=\frac{q^{2n+1}t+q^{-2n-1}t^{-1}}{(q-q^{-1})^2}=\frac{\pm q^{m+n+1}\pm q^{-m-n-1}}{(q-q^{-1})^2}.
	\end{align*}
	$(ii)$ Observe that $\chi_{m,n,\pm}(E^kF^kK^{-k})\neq 0$ if and only if $k\leq m$ by \eqref{quantum algebra: inequality 1} and \eqref{quantum algebra: equation 1}. Analogously, $\chi_{m,n,\pm}(F^kE^kK^k)\neq 0$ if and only if $k\leq n$ by \eqref{quantum algebra: inequality 2} and \eqref{quantum algebra: equation 2}. This implies that $\chi_{m,n,\pm}\in\cD_{-k}$ if and only if $-m\leq k\leq n.$ Now suppose $k\in\{0,1,\ldots,n\}$. Since $C_q$ commutes with $E,F$, we have
	\begin{align*}
		 \chi_{m,n,\pm}^k(K)&=\frac{\chi_{m,n,\pm}(E^{\ast k}KE^k)}{\chi_{m,n,\pm}(E^{\ast k}E^k)}=\frac{\chi_{m,n,\pm}(E^{\ast k}E^kq^{2k}K)}{\chi_{m,n,\pm}(E^{\ast k}E^k)}=q^{2k}\chi_{m,n,\pm}(K),\\
		\chi_{m,n,\pm}^k(C_q)&=\frac{\chi_{m,n,\pm}(E^{* k}C_qE^k)}{\chi_{m,n,\pm}(E^{* k}E^k)}=\chi_{m,n,\pm}(C_q).
	\end{align*}
	Analogously, if $k\in\{-m,-m+1,\ldots,0\}$ we have
	\begin{align*}
		 \chi_{m,n,\pm}^k(K)&=\frac{\chi_{m,n,\pm}(F^{\ast k}KF^k)}{\chi_{m,n,\pm}(F^{\ast k}F^k)}=q^{-2k}\chi_{m,n,\pm}(K),\\
		\chi_{m,n,\pm}^k(C_q)&=\chi_{m,n,\pm}(C_q).
	\end{align*}
	Hence, if $\chi_{m,n,\pm}^k$ is defined, then $\chi_{m,n,\pm}^k(K)=\pm q^{(m+k)-(n-k)}=\chi_{m+k,n-k,\pm}(K)$ and $\chi_{m,n,\pm}^k(C_q)=\chi_{m,n,\pm}(C_q)$.
\end{proof}

In particular, the previous proposition shows that for each $\chi\in\cBp$ the stabilizer $\St\chi$ is trivial. We set
$$
\Gamma:=\set{\chi_{0,n,+}\mid n\in\mathbb{N}_0}\cup\set{\chi_{n,-}\mid n\in\mathbb{N}_0}.
$$ 
As in Section \ref{sec: q-oscillator algebra}, we conclude that $\Gamma$ is a measurable countably separated section of the partial action. Using Proposition \ref{quantum algebra: set of positive characters} we conclude that $\textnormal{Orb}\chi_{0,n,\pm}$ consists of $n+1$ elements and hence $\textnormal{Ind}\chi_{0,n,\pm}$ has dimension $n+1$ by Proposition \ref{orthonormal base of representation space}, where $\chi_{0,n,\pm}\in\Gamma$. 
Now put $l:=\frac{n}{2}$ and $\pi_{\omega,l}:=\textnormal{Ind}\chi_{0,n,\pm}$, where $\omega=\pm 1$. Let $\{e_{l+m}\}_{m=-l,-l+1,\ldots,l}$ be an orthonormal base of the representation space $\mathcal{H}_{\pi_{l,\pm}}$ of $\pi_{l,\pm}$. For notational convenience, we put $e_{l+1}:=0,$ and $e_{-l-1}:=0$.
	
Using Proposition \ref{orthonormal base of representation space}, relations \eqref{quantum algebra: polynomial relations}, Proposition \ref{quantum algebra: set of positive characters} and the facts that $\chi_{0,n,\pm}(EF)=0$, $\chi_{0,n,\pm}([K;l+m])=-\omega[l-m]$, we obtain the action of $\pi_{\omega,l}$ on the base vectors $e_{l+m}$.
\begin{align*}
\pi_{\omega,l}(K)e_{l+m}&=\chi_{0,n,\pm}^{l+m}(K)e_{l+m}=\chi_{l+m,n-l-m,\pm}(K)e_{l+m}=\omega q^{2m}e_{l+m},\\
\pi_{\omega,l}(E)e_{l+m}&=\frac{\chi_{0,n,\pm}(E^{*(l+m+1)}E^{l+m+1})}{\left(\chi_{0,n,\pm}(E^{*(l+m+1)}E^{l+m+1})\chi_{0,n,\pm}(E^{*(l+m)}E^{l+m})\right)^{1/2}}e_{l+m+1}\\
&=\left(\frac{\chi_{0,n,\pm}(E^{*(l+m+1)}E^{l+m+1})}{\chi_{0,n,\pm}(E^{*(l+m)}E^{l+m})}\right)^{1/2}e_{l+m+1}\\
&=\left( q^{(l+m+1)(l+m+2)-(l+m)(l+m+1)}\right)^{1/2}\\
&\quad \times\left(\chi_{0,n,\pm}(EF-[l+m+1][K;l+m])\chi_{0,n,\pm}(K) \right)^{1/2}e_{l+m+1}\\
&=q^{m+1}\sqrt{[l-m][l+m+1]}e_{l+m+1},\\
\pi_{\omega,l}(F)e_{l+m}&=\pi_{\omega,l}(E^*K^{-1})e_{l+m}=\omega q^{-2m}\pi_{\omega,l}(E)^*e_{l+m}\\
&=\omega q^{-m}\sqrt{[l+m][l-m+1]}e_{l+m-1}.
\end{align*}
Putting $\gee_m:=e_{l+m},\ m=-l,\dots,l,$ we see that all irreducible well-behaved \sloppy$\ast$-representations of the quantum algebra $\mathcal{U}_q(su(2))$ are unitarily equivalent to the irreducible well-behaved \sloppy$\ast$-representation $\pi_{\omega,l}$, given by the formulas \eqref{eq_formulas}, for some $\omega\in\{-1,+1\}$ and $l\in \frac{1}{2}\dN_0$. In particular, all irreducible \sloppy$\ast$-representations of $\mathcal{U}_q(su(2))$ are bounded. Summarizing the above discussion, we obtain the following

\begin{thm}
	Every irreducible well-behaved $*$-representation of $\mathcal{U}_q(su(2)),\ q\in\mathbb{R}^+\backslash\{1\},$ is induced from a one-dimensional $*$-representation.
\end{thm}

Similarly to Lemma \ref{lem_pospol_qccr} and Theorem \ref{thm_badrep_qccr} one can prove the following Lemma and Theorem.
\begin{lem}
	The polynomial $$(EF-[2][K;1])(EF-[3][K;2])\in\mathbb{C}[EF,K,K^{-1}]$$ is positive in every well-behaved $*$-representation of $\cA$ and is not of the form $\sum_{k=1}^na_k^*a_k$ for $a_k\in\cA$.
\end{lem}

\begin{thm}
	There exists a $*$-representation of $\mathcal{U}_q(su(2)),\ q\in\mathbb{R}^+\backslash\{1\}$, which has no well-behaved extension in a possibly larger Hilbert space.
\end{thm}

Let $\Rep\cA$ denote the category of well-behaved non-degenerate representations of $\cA.$ Then $\cA,$ considered with $\Rep\cA$ and generators $E,F,K,K^{-1},$ has a $C^*$-envelope $\gA$ in the sense of Definition \ref{defn_envelope}. As in Section \ref{sec: q-oscillator algebra}, let $(C_0(\cBp),\dZ,\beta)$ be the $C^*$-p.d.s. dual to $(\cBp,\dZ,\alpha).$ The description of the p.d.s. $(\cBp,\dZ,\alpha)$ in Proposition \ref{quantum algebra: set of positive characters} implies that the $C^*$-p.d.s. $(C_0(\cBp),\dZ,\beta)$ is defined by the partial automorphism $\Theta=(\theta,I,J),$ where $I=I_{-1},\ J=I_1,\ (\theta(f))(t)=(\beta_1(f))(t),\ f\in I_{-1}.$

The proof of the following theorem is completely analogous to the proof of Theorem \ref{thm_envelope_qccr}.
\begin{thm}\label{thm_envelope_uqsu2}
Consider the $q$-deformed enveloping algebra $\cA=\cU_q(su(2))$ with generators $E,F,K,K^{-1}$ and the category of well-behaved representations $\Rep\cA.$ Then the covariance algebra $\gA:=C^*(C_0(\cBp),\Theta)$ is a $C^*$-envelope of $\cA.$
\end{thm}


\end{document}